\documentclass[11pt,reqno]{amsart}
\usepackage{latexsym,amssymb,amsmath,multicol,rotating, lscape}

\makeatletter
\@namedef{subjclassname@2020}{%
  \textup{2020} Mathematics Subject Classification}
\makeatother

 2

\usepackage{longtable}
\newtheorem{thm}{Theorem}[section]

\newtheorem{cor}[thm]{Corollary}
\newcommand{\thmref}[1]{Theorem~\ref{#1}}

\newcommand{\corref}[1]{Corollary~\ref{#1}}

\theoremstyle{remark}
\newtheorem{rmk}{Remark}[section]

\begin{document}

\title[extension of Ramanujan-Serre derivative map]
{A Simple Extension of Ramanujan-Serre Derivative Map and Some Applications}
 \author{B. Ramakrishnan, Brundaban Sahu and Anup Kumar Singh}
\address[B. Ramakrishnan]{Indian Statistical Institute, North-East Centre, Punioni, Solmara\\ 
Tezpur - 784 501, Assam, India} 
\address[Brundaban Sahu and Anup Kumar Singh]
{School of Mathematical Sciences, National Institute of Science Education and Research, An OCC of  Homi Bhabha National Institute, Bhubaneswar,  
Via: Jatni, Khurda, Odisha - 752 050, India.}
\email[B. Ramakrishnan]{ramki@isine.ac.in, b.ramki61@gmail.com}
\email[Brundaban Sahu]{brundaban.sahu@niser.ac.in}
\email[Anup Kumar Singh]{anupsinghmath@niser.ac.in, anupsinghmath@gmail.com}



\begin{abstract} 
If $f(z)$ is a modular form of weight $k$, then the differential operator $\vartheta_k$ defined by 
$\vartheta_k(f) = \frac{1}{2\pi i} \frac{d}{dz}f(z) - \frac{k}{12} E_2(z) f(z)$ (known as the Ramanujan-Serre derivative map) 
is a modular form of weight $k+2$. In this paper, we obtain a simple extension of this map and use it to get 
a general method to derive certain convolution sums of the divisor functions (using the theory of modular forms). 
Explicit expressions are given for four types of convolution sums and we provide many examples for all these types 
of sums.
\end{abstract}


\maketitle
\section{Introduction}
Let $\mathbb N$ denote the set of positive integers. For $r, n\in \mathbb N$, let 
$\sigma_r(n) = \displaystyle{\sum_{d\vert n, d\in \mathbb N} d^r}$ be the divisor function. If $n$ is not a positive integer, set $\sigma_r(n) =0$ 
and we write $\sigma(n)$ for $\sigma_1(n)$. 
For $a,b, r,s, n \in \mathbb N$, we define $W^{r, s}_{a,b}(n)$ by 
\begin{equation}
W^{r, s}_{a,b}(n) := \displaystyle{\sum_{l,m \in \mathbb N \atop{al+bm=n}}} \sigma_r(l)\sigma_s(m).
\end{equation}
These sums are referred to as the convolution sums of the divisor functions. When $r=s=1$, it is denoted simply by $W_{a, b}(n)$. Further, we write $W_{1,a}(n) = W_{a,1}(n) = W_a(n)$. Evaluation of these sums has a long history, going back to the works of Besge, Glaisher and Ramanujan \cite{{besge}, {glaisher}, {ramanujan}}. In the literature there are various methods used to obtain these convolution sums namely, elementary evaluation, using the theory of modular forms and quasimodular forms and also using $(p,k)$ parametrization etc. We refer to the book by K. S. Williams \cite{williams-book} for more details about the history of this problem. In this paper, we look at a basic result in the theory of modular forms and present a simple extension. 

Let $f(z)$ be a modular form of weight $k$ for the full modular group and let $E_2(z)$ be the weight $2$ Eisenstein series, which is a quasimodular form. An easy computation shows that the function $\frac{1}{2\pi i} \frac{d}{dz}f(z) - \frac{k}{12} E_2(z) f(z)$ is a modular form of weight $k+2$ for the full modular group.
The result also holds when $f$ is a modular form with level (i.e., with respect to the congruence subgroup $\Gamma_0(N)$). 
If $f$ is a cusp form, then the resulting function is also a cusp form. Let us denote this map by $\vartheta_k$ and write 
$$
\vartheta_k(f)(z) = Df(z) - \frac{k}{12} E_2(z)f(z),
$$
where $D = \frac{1}{2\pi i} \frac{d}{dz}$. This is called as the Ramanujan-Serre derivative map. 

By taking the logarithmic derivative of  $\Delta(z)$, the normalized cusp form of weight $12$, Ramanujan \cite{ramanujan} derived the following recursion formula for $\tau(n)$, the $n$-th Fourier coefficient of $\Delta(z)$ (for $n\ge 2$):
\begin{equation}\label{tau}
\tau(n) = \frac{24}{1-n}\sum_{m=1}^{n-1} \sigma(m) \tau(n-m).
\end{equation}
Note that applying $\vartheta_{12}$ on $\Delta(z)$ gives a cusp form of weight $14$ on $SL(2, {\mathbb Z})$ and therefore it must be zero. i,e., $D\Delta(z) = E_2(z) \Delta(z)$. Therefore, the above recursion formula also follows from this fact.  
In the same work \cite{ramanujan}, Ramanujan extensively studied the convolution sums $W_{1,1}^{r,s}(n)$, where 
$r, s$ are odd positive integers. He expressed $W_{1,1}^{r,s}(n)$ as a sum of a main term and an error term, where the main term (arising from Eisenstein series) was given explicitly. The error term corresponds to the coefficients coming from cusp forms. For some cases of $(r, s)$ he also obtained explicit expressions for these convolution sums. 

In recent years, there have been many works on explicit evaluation of these convolution sums using various methods: 
elementary evaluation using some combinatorial results, using the theory of modular forms or quasimodular forms and 
using the $(p,k)$ parametrization etc. In this paper, we observe that the Ramanujan-Serre differential operator can be extended in a simple way which can be used to evaluate the convolution sums in which one of the divisor functions is 
$\sigma(n)$, i.e., when one of the integers $r$ or $s$ is equal to $1$. If both are odd integers greater than 1, then 
$E_{r+1}(az)$ and $E_{s+1}(bz)$ are modular forms  (see \eqref{ek} for the definition) of weights $r+1$ and $s+1$ 
respectively and so theproduct $E_{r+1}(az) E_{s+1}(bz)$ is a modular form of weight $r+s+2$ for the group $\Gamma_0(ab)$. 
The required convolution sum $W_{a,b}^{r,s}(n)$ follows by expressing this product as a linear combination of basis elements and then comparing the $n$-th Fourier coefficients.  So, our method in this work is needed only when one of the numbers $r$ or $s$ is equal to 1. Applying the extended Ramanujan-Serre derivative map on modular forms (which involves the Eisenstein series $E_2(z)$), gives rise to a modular form and so one can use a basis for the space of modular forms for the purpose of evaluating the corresponding convolution sums. 

Ramanujan also proved that $E_2^2(z) - 12 DE_2(z)$ is a modular form of weight $4$, which implies that $E_2^2(z) - 12 DE_2(z) = E_4(z)$. In our work, we also give a simple extension of this result, from which one can 
evaluate the convolution sums $W_{a,b}(n)$ for any $a,b$ in terms of a basis for the space of modular forms of weight $4$ and level $ab$ (here it is enough to consider the case when $\gcd(a,b)=1$).   The advantage of our method is that it avoids the use 
of the structure of quasimodular forms space in evaluating convolution sums involving the divisor function $\sigma(n)$, which is (up to a constant) the $n$-th Fourier coefficient of the quasimodular form $E_2(z)$. 

More precisely, first we show that for positive integers $a, b$, the differential operator $\vartheta_{k;(a,b)}$ defined by 

\begin{equation}\label{theta}
\vartheta_{k;(a,b)}(f) :=  b Df(bz) - \frac{k}{12} a E_2(az) f(bz)
\end{equation}
maps the space $M_k(M,\chi)$ into the space $M_{k+2}(N,\chi)$, where $N = l.c.m(Mb,a)$. (For the definition of the vector space of modular forms, we refer to \S 2.)

\begin{rmk}
To simplify the notation, we will be writing $b Df(bz)$ as ${\mathcal D}f(bz)$. If $f(z)$ has the Fourier expansion $\sum_{n\ge 0} a_f(n) e^{2\pi inz}$, then the Fourier expansion of ${\mathcal D}f(bz)$ is given by 
$$
{\mathcal D}f(bz)  = \sum_{n\ge 0} n ~a_f(n/b) e^{2\pi inz} = b Df(bz).
$$
We also observe that the effect of the  operator ${\mathcal D}$ on the Fourier expansion  is consistent with the SAGE calculations. When $b=1$, both the operators $D$ and ${\mathcal D}$ are the same. 
\end{rmk}

In view of the above remark, we modify the definition of $\vartheta_{k;(a,b)}$ as follows: 

\begin{equation*}
\hskip 3.5cm \vartheta_{k;(a,b)}(f) =   {\mathcal D}f(bz) - \frac{k}{12} a E_2(az) f(bz). \hskip 4cm  \qquad \quad \hfil (3)'
\end{equation*}
When $a=b=1$, then $\vartheta_{k;(1,1)}$ is nothing but $\vartheta_k$. 
In particular, by taking $f(z) = E_k(z)$, the weight $k$ Eisenstein series for the full modular group, one obtains the convolution sums $W_{a,b}^{1,k-1}(n)$ using a basis for the space $M_{k+2}(N)$. In section \S 4.1, we give some examples using our main result. 

Our next result is about extending the fact that $E_2^2(z) - 12 DE_2(z)$ is a modular form of weight $4$. 
When $a, b$ are relatively prime, we show that the function 

\begin{equation}\label{e2}
 \frac{6a}{b} DE_2(az) + \frac{6b}{a} DE_2(bz) - E_2(az) E_2(bz)
\end{equation}
is a modular form in the space $M_4(ab)$. In terms of the operator ${\mathcal D}$, the above is equivalent to 
\begin{equation*}
\hskip 4.5cm \frac{6}{b} {\mathcal D}E_2(az) + \frac{6}{a} {\mathcal D}E_2(bz) - E_2(az) E_2(bz). \hskip 4.5cm (4)'
\end{equation*}

As an application, we obtain the convolution sum 
$W_{a,b}(n)$ by using a basis for the space $M_4(ab)$. We provide some examples demonstrating our method in section 
\S 4.2.  
We remark here that in \cite{aygin}, Aygin also evaluated the convolution sums $W_{1,p}(n)$, $W_{p_1,p_2}(n)$, $W_{1,p_1p_2}(n)$, where $p, p_1,p_2$ are prime numbers. The method used in this work is different from our method, though both of these works make use of the fact that the resulting function is a modular form. In \cite{aygin}, the author used the fact that $E_2(z) - t E_2(tz)$ is a modular form of weight $2$ on $\Gamma_0(t)$, where $t\in {\mathbb N}$. It is also to be noted that in \cite{aygin}, the author  obtained the required convolution sums by using  the already known sum $W_{1}(n)$.

In \S 3, we use the operator $\vartheta_{k+2;(a,1)}$ on the modular form of weight $k+2$ given by \eqref{theta} (with $f(z) = E_k(z)$) and the operator $\vartheta_{4;(a,1)}$ on the modular form of weight $4$ given by \eqref{e2} to get 
explicit expressions for the following convolution sums:
\begin{equation}
\sum_{l,m \in {\mathbb N}\atop{al+bm=n}} l \sigma(l)\sigma_{k-1}(m);\quad \sum_{l,m \in {\mathbb N}\atop{al+bm=n}} l \sigma(l)\sigma(m),
\end{equation}
respectively. We give some explicit examples in \S 4.3 and \S4.4.

\bigskip

\section{Preliminaries and main results}

Let $k,M\in {\mathbb N}$ and $M_k(M,\chi)$ (resp. $S_k(M,\chi)$) denote the finite dimensional vector space of modular forms (resp. cusp forms) of weight $k$ on $\Gamma_0(M)$, where $\chi$ is a Dirichlet character such that $\chi(-1) = (-1)^k$. When $M=1$, the space is denoted as $M_k$ (resp. $S_k$) and when $\chi$ is the principal character, the spaces are denoted as $M_k(M)$ and $S_k(M)$ respectively. The dimension of the vector space $M_k(M)$ is denoted by $\lambda_k(M)$.  
For an even integer $k\ge 4$, let 
\begin{equation}\label{ek}
E_k(z) = 1 - \frac{2k}{B_k} \sum_{n\ge 1} \sigma_{k-1}(n) q^n,
\end{equation}
be the Eisenstein series in $M_k$, where $q = e^{2\pi iz}$ and $B_k$ is the $k$-th Bernoulli number. For $k=2$, 
we denote the weight $2$ Eisenstein series by $E_2(z)$, which is given by 
\begin{equation}\label{quasi}
E_2(z) = 1 - 24 \sum_{n\ge 1} \sigma(n) q^n.
\end{equation}
It is a quasimodular form for the full modular group. For the basic theory of modular forms, we refer to \cite{serre}. 
 
\smallskip
The following is the main result of this paper. 

\begin{thm}\label{main}
Let $a, b$ be positive integers and let $f$ be a modular form in $M_k(M,\chi)$, where 
$k,M$ are positive integers, $\chi$ is a Dirichlet character modulo $M$ such that $\chi(-1) = (-1)^k$. 
Then $\vartheta_{k;(a,b)}(f)$ is a modular form in $M_{k+2}(N, \chi)$, where $N = l.c.m (Mb,a)$ and $\vartheta_{k;(a,b)}$ is the extended operator defined by \eqref{theta}$'$. Furthermore, if $a,b$  are relatively prime, the function 
$\frac{6a}{b} DE_2(az) + \frac{6b}{a} DE_2(bz) -E_2(az) E_2(bz) =  \frac{6}{b} {\mathcal D}E_2(az) + \frac{6}{a} {\mathcal D}E_2(bz) - E_2(az) E_2(bz)$ is a modular form in the space $M_4(ab)$. 
\end{thm}

Before proceeding to the proof of this theorem, we shall make some observations and deduce some applications. 

\begin{rmk}
As observed earlier, when $a = b =1$, the differential operator $\vartheta_{k;(1,1)}$ is nothing but $\vartheta_k$. In this case, our second result in the above theorem is nothing but the fact that $12 DE_2(z) - E_2^2(z) = - E_4(z)$, as proved by Ramanujan.  Though there is no need to assume relatively prime condition in the second part of the theorem, for the application part, it is 
sufficient to assume this condition. 
\end{rmk}

\smallskip

Let $\{f_i(z)| 1\le i\le \lambda_{k+2}(ab)\}$ be a basis of the vector space $M_{k+2}(ab)$ and denote the $n$-th Fourier coefficient of $f_i(z)$ by $a_i(n)$. Here $a, b$ are relatively prime positive integers. Using our main theorem, we obtain the following corollaries. By taking $f(z) = E_k(z)$ (with $k\ge 4$ and an even integer) and assuming that  $a$ and $b$ are relatively prime positive integers in the above theorem, we obtain the convolution sum $W_{a,b}^{1,k-1}(n)$ in terms of the $a_i(n)$.  

\begin{cor}\label{1-k-1}
Let $k\ge 4$ be an even integer and $f(z) =E_k(z)$. If $a, b$ are relatively prime positive integers, then there exist 
constants $\alpha_i\in {\mathbb C}$, $1\le i\le \lambda_{k+2}(ab)$ such that 
\begin{equation}\label{1:k-1}
W_{a,b}^{1,k-1}(n) = \frac{B_k}{2k} \sigma(n/a) +  \left(\frac{1}{24} - \frac{n}{2ka}\right) \sigma_{k-1}(n/b) - \frac{B_k}{4ak^2}\sum_{i=1}^{\lambda_{k+2}(ab)} \alpha_i a_i(n),
\end{equation}
where $B_k$ is the $k$-th Bernoulli number and $a_i(n)$, $1\le i\le \lambda_{k+2}(ab)$ is the $n$-th Fourier coefficient 
of the $i$-th basis element $f_i$ of $M_{k+2}(ab)$ as denoted above. 
\end{cor}

\begin{proof}
When $f(z) = E_k(z)$, then the function ${\mathcal D}E_k(bz) - \frac{k}{12} a E_2(az) E_k(bz)$ belongs to $M_{k+2}(ab)$. 
Let  us denote a basis of $M_{k+2}(ab)$ by $f_i(z)$, $1\le i\le \lambda_{k+2}(ab)$, whose $n$-th Fourier coefficients are given by $a_i(n)$. Then the above modular form can be expressed as a linear combination of these basis elements.  Therefore, there exists constants $\alpha_i$, $1\le i\le \lambda_{k+2}(ab)$ such that 
\begin{equation*}
{\mathcal D}E_k(bz) - \frac{k}{12} a E_2(az) E_k(bz) = \sum_{i=1}^{\lambda_{k+2}(ab)} \alpha_i f_i(z).
\end{equation*}
Now,  by comparing the $n$-th Fourier coefficients on both sides of the above identity, we get 
\begin{equation*}
-\frac{2k}{B_k} n\sigma_{k-1}(n/b) +2ka \sigma(n/a) + \frac{k^2a}{6 B_k} \sigma_{k-1}(n/b) - \frac{4k^2 a}{B_k} W_{a,b}^{1,k-1}(n) = \sum_{i=1}^{\lambda_{k+2}(ab)} \alpha_i a_i(n).
\end{equation*}
Simplifying, we get 
\begin{equation*}
W_{a,b}^{1,k-1}(n) = \frac{B_k}{2k} \sigma(n/a) + \big(\frac{1}{24} - \frac{n}{2ka}\big) \sigma_{k-1}(n/b) - \frac{B_k}{4ak^2}\sum_{i}^{\lambda_{k+2}(ab)} \alpha_i a_i(n).
\end{equation*}
\end{proof}

Let $g_i(z)$, $1\le i\le \lambda_4(ab)$ denote a basis of $M_4(ab)$. 
By using the second part of the above theorem, expressing \eqref{e2} as a linear combination of $g_i(z)$ and comparing the $n$-th Fourier coefficients, we obtain the following corollary. 

\begin{cor}\label{wab}
If $a, b$ are relatively prime positive integers, then there exist constants $\beta_i\in {\mathbb C}$, $1\le i\le \lambda_{4}(ab)$ such that 
\begin{equation}
W_{a,b}(n) = \frac{1}{24}\big(1-\frac{6n}{b}\big)\sigma(n/a) + \frac{1}{24}\big(1-\frac{6n}{a}\big) \sigma(n/b) - \frac{1}{576} \sum_{i=1}^{\lambda_{4}(ab)} \beta_i b_i(n),
\end{equation}
where $b_i(n)$, $1\le i\le \lambda_4(ab)$ is the $n$-th Fourier coefficient of the $i$-th basis element $g_i$ of $M_4(ab)$ as denoted above. 
\end{cor}

\bigskip

\noindent {\bf Proof of \thmref{main}}. Let $N = l.c.m(Mb, a)$ and let $A = \begin{pmatrix}\alpha & \beta \\
\gamma & \delta\end{pmatrix} \in \Gamma_0(N)$. Then we have $A' = \begin{pmatrix} \alpha & \beta b\\ \gamma/b & \delta\\
\end{pmatrix} \in \Gamma_0(M)$ and $A'' = \begin{pmatrix} \alpha & \beta a\\ \gamma/a & \delta\\
\end{pmatrix} \in SL(2,{\mathbb Z})$. Since $f\in M_k(M,\chi)$, we have the following transformation:\\
\begin{equation}\label{fbz}
f(b Az) = f(A'(bz)) = \chi(\delta) (\gamma z+\delta)^k f(bz).
\end{equation}
Differentiating \eqref{fbz} with respect to $z$ and using the fact that $D = \frac{1}{2\pi i} \frac{d}{dz}$, we get 
\begin{equation}\label{dfbz}
b Df(b Az) = \chi(\delta) (\gamma z+\delta)^{k+2}  b Df(bz) + \chi(\delta) \frac{k\gamma}{2\pi i} (\gamma z+\delta)^{k+1} f(bz).
\end{equation}
The transformation formula for $E_2(az)$ is given by 
\begin{equation}\label{e2az}
E_2(a Az) = E_2(A''(az)) = (\gamma z+\delta)^2 E_2(az) + \frac{12}{2\pi i} \frac{\gamma}{a} (\gamma z+ \delta).
\end{equation}
Using equations (\ref{fbz}), (\ref{dfbz}) and (\ref{e2az}), it follows easily that 
\begin{equation}
b Df(b Az) - \frac{ka}{12} E_2(aAz) f(b Az) = \chi(\delta) (\gamma z+\delta)^{k+2} \big(b Df(bz) - \frac{ka}{12} E_2(az) f(bz)\big),
\end{equation}
from which it follows that the function $b Df(bz) - \frac{ka}{12} E_2(az) f(bz)$ is a modular form and belongs to 
$M_{k+2}(N, \chi)$. 

To get the second part of the theorem, we take $A\in \Gamma_0(N)$, $N=ab$. 
Differentiating  \eqref{e2az} with respect to $z$, we get 
\begin{equation}
a DE_2(aAz) = (\gamma z+\delta)^4 a DE_2(az) + \frac{2\gamma}{2\pi i} (\gamma z+\delta)^3 E_2(az) + \frac{12}{a} \big(\frac{\gamma}{2\pi i}\big)^2 (\gamma z+ \delta)^2,
\end{equation} 
and a similar identity replacing $a$ by $b$. Using these two identities, it is easy to see that 
\begin{equation}
\begin{split}
& \frac{6a}{b} DE_2(aAz) + \frac{6b}{a} DE_2(bAz) -E_2(aAz) E_2(bAz) \hskip 6cm \\
&\hskip 3cm = (\gamma z+\delta)^4 \left(\frac{6a}{b} DE_2(az) + \frac{6b}{a} DE_2(bz) -E_2(az) E_2(bz)\right),
\end{split}
\end{equation}
from which the modular property follows. This completes the proof. 
$~$\hfill \qed

\smallskip

\section{Convolution sums of another type}

As an application of the extended Ramanujan-Serre derivative map, one can evaluate another type of convolution sum, which is defined as follows. Let $e, r, s$ be natural numbers with $r, s$ being odd integers. For natural numbers $a, b$ with $\gcd(a,b)=1$, define a sum as follows. 
\begin{equation}\label{ers}
W_{a,b}^{e;r,s}(n) := \sum_{l.m\in {\mathbb N}\atop{al+bm=n}} l^e \sigma_r(l) \sigma_s(m).
\end{equation}

\bigskip

In this section, we use the  operator $\vartheta_{k+2;(a,1)}$ on the modular form defined by \eqref{theta} (with $f(z) =E_k(z)$) and the operator $\vartheta_{4;(a,1)}$ on the modular form defined by \eqref{e2} to 
evaluate the convolution sums $W_{a,b}^{e;r,s}(n)$ with $e=r=1, s=k-1$ and $e=r=s=1$ respectively.
We obtain these applications in the next two subsections. 

\bigskip

\subsection{Evaluation of the convolution sum $\sum_{al+bm=n}l\sigma(l)\sigma_{k-1}(m)$} ~ 
By \thmref{main}, the function ${\mathcal D}E_k(bz) - \frac{k}{12} a E_2(az) E_k(bz)$ is a modular form in $M_{k+2}(ab)$, where $a$ and $b$ are relatively prime positive integers. Now applying the operator $\vartheta_{k+2;(a,1)}$ on this function, we get (by noting that $D = {\mathcal D}$ as $b=1$),
$$
{\mathcal D}\big({\mathcal D}E_k(bz) - \frac{k}{12} a E_2(az) E_k(bz)\big) - \frac{(k+2)}{12}a E_2(az) \big({\mathcal D}E_k(bz) - 
\frac{k}{12} a E_2(az) E_k(bz)\big),
$$
and by \thmref{main} it belongs to $M_{k+4}(ab)$. To simplify the above function we use the fact that $E_2^2(az) = \frac{12}{a}  {\mathcal D}E_2(az) + E_4(az)$ and also the fact that $E_4(az) E_k(bz)\in M_{k+4}(ab)$ to conclude that the following function is a modular form. In fact, we have 
$$
{\mathcal D}^2E_k(bz) + \frac{ka}{12} (k+1) {\mathcal D}E_2(az) E_k(bz) - \frac{a(k+1)}{6} E_2(az) {\mathcal D}E_k(bz) \in M_{k+4}(ab).
$$
(Note: For an integer $r\ge 1$, we have ${\mathcal D}^r f(bz) = \sum_{n\ge 0} n^r a_f(n/b) q^n$.) Let $h_i(z)$, $1\le i\le \lambda_{k+4}(ab)$ be a basis for the space $M_{k+4}(ab)$. Therefore, there exists constants 
$\gamma_i$, $1\le i\le \lambda_{k+4}(ab)$ such that 
\begin{equation*}
 {\mathcal D}^2E_k(bz) + \frac{ka}{12} (k+1) {\mathcal D}E_2(az) E_k(bz) - \frac{a(k+1)}{6} E_2(az) {\mathcal D}E_k(bz) 
= \sum_{I=1}^{\lambda_{k+4}(ab)} \gamma_i h_i(z).
\end{equation*}

\noindent 
Let us denote the $n$-th Fourier coefficient of the $i$-th basis element $h_i(z)$  as $c_i(n)$. Then by comparing the $n$-th Fourier coefficients on both the sides of the above identity and using the fact that 
\begin{equation}
b \sum_{al+bm=n} m\sigma(l)\sigma_{k-1}(m) = n W_{a,b}^{1,k-1}(n) - a \sum_{al+bm=n} l\sigma(l) \sigma_{k-1}(m),
\end{equation} 
we obtain the following theorem, where $W_{a,b}^{1,k-1}(n)$ is given by \corref{1-k-1}.

\smallskip

\begin{thm}\label{3.4}
Let $k\ge 4$ be an even integer and $a,b$ be positive integers such that $\gcd(a,b)=1$. Then for $n\ge 2$, there exists constants $\gamma_i$, $1\le i\le \lambda_{k+4}(ab)$, such that 
\begin{equation}
\begin{split}
\sum_{l,m\in {\mathbb N}\atop{al+bm=n}} l\sigma(l) \sigma_{k-1}(m) & = \frac{6n^2 - a(k+1)n}{12 a^2(k+1)(k+2)} 
\sigma_{k-1}(n/b) + \frac{B_k}{2a(k+2)} n\sigma(n/a) \quad \\
&~~\quad + \frac{2n}{a(k+2)} W_{a,b}^{1,k-1}(n) + \frac{B_k}{4a^2k(k+1)(k+2)}\!\!\!\sum_{i=1}^{\lambda_{k+4}(ab)} \!\!\!\gamma_i c_i(n),
\end{split}
\end{equation}
where $B_k$ is the $k$-th Bernoulli number, $\lambda_{k+4}(ab)$ is the dimension of the space $M_{k+4}(ab)$. 
The convolution sum $W_{a,b}^{1,k-1}(n)$ is given by \corref{1-k-1} and $c_i(n)$ is the $n$-th Fourier coefficient of the $i$-th  basis element $h_i$ of $M_{k+4}(ab)$ as denoted above. 
\end{thm}

\smallskip

\begin{rmk}
We provide some examples for the above convolution sums in \S 4.3 by taking particular values of $a, b, k$.
 \end{rmk}
 
 \bigskip
 
\subsection{Evaluation of the convolution sum $\sum_{al+bm=n}l\sigma(l)\sigma(m)$} ~We now make use of the weight 4 modular form given by \eqref{e2} and apply the operator $\vartheta_{4;(a,1)}$ on this function to get a formula for the convolution sum $\sum_{al+bm=n} l\sigma(l) \sigma(m)$. Here again, $D = {\mathcal D}$ as $b=1$. 
In fact, this implies that the function 
\begin{equation*}
{\mathcal D}\Big(\frac{6}{b} {\mathcal D}E_2(az) + \frac{6}{a} {\mathcal D}E_2(bz) -E_2(az) E_2(bz)\Big) -\frac{a}{3} E_2(az) \Big(\frac{6}{b} {\mathcal D}E_2(az) + \frac{6}{a} {\mathcal D}E_2(bz) -E_2(az) E_2(bz)\Big)
\end{equation*}
is a modular form in $M_6(ab)$. Simplifying, we see that the function
\begin{equation*}
\frac{6}{b}{\mathcal D}^2E_2(az) + \frac{6}{a}{\mathcal D}^2E_2(bz) - E_2(az)\big[3{\mathcal D}E_2(bz) + \frac{2a}{b}{\mathcal D}E_2(az)\big] + E_2(bz)\big[3{\mathcal D}E_2(az) + \frac{a}{3}E_4(az)\big]
\end{equation*}
belongs to $M_6(ab)$, where $a, b$ are positive integers with $\gcd(a,b)=1$.  Let $\{F_i(z): 1\le i\le \lambda_6(ab)\}$ be a basis for the space $M_6(ab)$ whose $n$-th Fourier coefficients are denoted by $A_i(n)$, $1\le i\le \lambda_6(ab)$. 
Expressing the above modular form in terms of the basis elements $F_i$, and comparing the $n$-th Fourier coefficients, we obtain the following theorem. 
\begin{thm}\label{3.5}
Let  $a$ and $b$ be positive integers such that $\gcd(a,b)=1$ and assume that at least one of them is greater than 1. 
Then for $n\ge 2$, there exists constants $\delta_i$, $1\le i\le \lambda_{6}(ab)$, such that 
\begin{equation}
\begin{split}
\sum_{al+bm = n} l\sigma(l)\sigma(m) & = \frac{1}{144a}\Big(\frac{6}{b}n^2 +3n - \frac{2an}{b}\Big) \sigma(n/a)  +
\frac{1}{144a}\Big(\frac{6}{a}n^2 -3n + \frac{a}{3}\Big) \sigma(n/b) - \frac{5}{216} \sigma_3(n/a) \\
& \quad +\frac{n}{2a} W_{a,b}(n) + \frac{5}{9} W_{b,a}^{1,3}(n) + \frac{a}{3b}\sum_{l+m=n/a}l\sigma(l)\sigma(m) + \frac{1}{6\times 576 a} \sum_{i=1}^{\lambda_6(ab)} \delta_i A_i(n),
\end{split}
\end{equation}
where $A_i(n)$ are the $n$-th Fourier coefficients of a basis of $M_6(ab)$ (as denoted above), $W_{b,a}^{1,3}(n)$ is given by \corref{1-k-1} and the sum $\sum_{l+m=n}l\sigma(l)\sigma(m)$ is given by \eqref{1-1} (see the remark below).
\end{thm}

\smallskip

\begin{rmk}
When $a=b=1$, the convolution sum $\sum_{l+m=n}l\sigma(l)\sigma(m)$ is given by the following formula:
\begin{equation}\label{1-1}
\sum_{l+m=n}l\sigma(l)\sigma(m) = \frac{5}{24}n\sigma_3(n) -\frac{1}{24}(6n^2-n) \sigma(n)
\end{equation}
The above is obtained by differentiating the identity  $E_2^2(z) = 12DE_2(z) +E_4(z)$ and comparing the $n$-th Fourier coefficients.  In \S 4.4, we illustrate \thmref{3.5} with some examples. 
\end{rmk}

\begin{rmk}
The convolution sums defined by \eqref{ers} can be evaluated for $e\ge 2$ by using our extended theta operator. 
However, the resulting expression involves triple convolution sums as well (when $e=2$). So, to deal with the case 
$e=2$, one needs to find similar triple convolution sums (and higher convolution sums when $e>2$). 
In our forthcoming work, we consider the problem of evaluating convolution sums involving three or more divisor functions. 
\end{rmk}

\section{Examples}

In this section we shall provide explicit examples for the evaluation of the convolution sums obtained in Corollaries 
\ref{1-k-1}, \ref{wab} and Theorems \ref{3.4}, \ref{3.5}. For this purpose, we use explicit bases for the spaces of modular forms for a few values of $k, a, b$. In fact, once we know the explicit basis for the space of modular forms of weight $k$, level $N$, then all the formulas for the convolution sums derived in sections 2 and 3 can be obtained explicitly. 
 
The basis elements used here are expressed in terms of Eisenstein series (with their duplications) and newforms, which generate the space of cusp forms. As mentioned earlier, $E_k(z)$ given by \eqref{ek} denotes the normalized Eisenstein series of weight $k$ for the full modular group. We denote by $\Delta_{k,N}(z)$, the unique (normalized) newform of weight $k$, level $N$ and $\Delta_{k,N;j}(z)$ denotes the $j$-th (normalized) newform of weight $k$, level $N$, when the space of newforms has dimension $>1$. These newforms are ordered as per the list provided in LMFDB \cite{lmfdb}. 
The $n$-th Fourier coefficient of the unique newform is denoted by $\tau_{k,N}(n)$ and the $n$-th Fourier coefficient of the $j$-th newform of weight $k$, level $N$ is denoted as $\tau_{k,N;j}(n)$. Explicit expressions for the newforms used in our bases are presented in the Appendix. 

For the examples provided in this paper, we use the following basis for the space of modular forms $M_k(N)$ for the values of $k, N$ listed in the table below. 

\begin{center}

{\bf Table 1}. Basis for the space $M_k(N)$. 

\smallskip
\begin{tabular}{|c|c|c|c|}
\hline 
$k$ & $N$ & dim & Basis \\
\hline 
4&30&22& $\{E_4(t_1z)\!:\!t_1\vert 30, \Delta_{4,5}(t_2z)\!:\!t_2\vert 6, \Delta_{4,6}(t_3z)\!:\!t_3\vert 5,\Delta_{4,10}(t_4z) : t_4\vert 3,$\\
&&& $\Delta_{4,15;1}(t_5z), \Delta_{4,15;2}(t_6z)\!:\!t_5, t_6\vert 2, \Delta_{4,30;1}(z), \Delta_{4,30;2}(z)\}$\\
\hline 
6 & 4& 4& $\{E_6(z), E_6(2z), E_6(4z), \Delta_{6,4}(z)\}$\\
\hline 
 6 & 6& 7& $\{E_6(tz), t\vert 6, \Delta_{6,3}(z), \Delta_{6,3}(2z), \Delta_{6,6}(z)\}$\\
\hline
8 & 4 & 5&  $\{E_8(z), E_8(2z), E_8(4z), \Delta_{8,2}(z), \Delta_{8,2}(2z)\}$\\
\hline
8 & 6 & 9&  $\{E_8(tz), t\vert 6, \Delta_{8,2}(z), \Delta_{8,2}(3z), \Delta_{8,3}(z), \Delta_{8,3}(2z), \Delta_{8,6}(z)\}$\\
\hline
10 & 4 & 6&  $\{E_{10}(z), E_{10}(2z), E_{10}(4z), \Delta_{10,2}(z), \Delta_{10,2}(2z), \Delta_{10,4}(z)\}$\\
\hline
10 & 6 & 11 &$\{E_{10}(tz), t\vert 6; \Delta_{10,2}(z), \Delta_{10,2}(2z), \Delta_{10,3;1}(z),$\\
 &  &  & $ \Delta_{10,3;1}(2z), \Delta_{10,3;2}(z), \Delta_{10,3;2}(2z), \Delta_{10,6}(z)\}$\\
\hline 
12&4&7 & $\{E_{12}(z), E_{12}(2z), E_{12}(4z), \Delta(z), \Delta(2z), \Delta(4z), \Delta_{12,4}(z)\}$\\
\hline 
12& 6 & 13 &$\{E_{12}(t_1z), t_1\vert 6; \Delta(t_2z), t_2\vert 6; \Delta_{12,3}(z), \Delta_{12,3}(2z),$ \\
 &  &  & $ \Delta_{12,6;1}(z), \Delta_{12,6;2}(z), \Delta_{12,6;3}(z)\}$ \\
\hline 
\end{tabular}
\end{center}

\smallskip

Note that when $N=2$ or $3$, we get a basis for $M_k(N)$ as a subset of the above basis for $M_k(4)$ or $M_k(6)$. 
So, when the level is 2 or 3, we make use of this fact. We also use a similar fact when we consider the space $M_4(d)$, 
where $d\vert 30$ for the convolution sums $W_{a,b}(n)$. 

\subsection{Examples for the convolution sums $W_{a,b}^{1,k-1}(n)$}

The following table gives some known results for the evaluation of this type of convolution sum. Here we list when 
$k=4, 6, 8, 10, 12$. 

\begin{center}
{\bf Table 2}. Known convolution sums $W_{a,b}^{1,k-1}(n)$. 
\smallskip

\begin{tabular}{|c|c|c|}
\hline
$k$ & $(a,b)$ & References\\
\hline
&$(1,1)$ & \cite{{ramanujan}, {lahiri1},{lahiri2}, {cw2}, {royer}}\\
\cline{2-3}
&$(1,2), (2,1)$ &\cite{{huard}, {cw2},{melfi1},{melfi2}, {royer}} \\
\cline{2-3}
4&$(1,3), (3,1)$ &\cite{yao}\\
\cline{2-3}
&$(1,4), (4,1)$ & \cite{cw2}\\
\cline{2-3}
&$(1,6), (6,1)$, $(2,3)$, $(3,2)$ &\cite{kokluce}\\
\hline
&$(1,1)$ & \cite{{ramanujan}, {lahiri1},{lahiri2}, {royer}} \\
\cline{2-3}
6&$(1,2), (2,1), (1,4), (4,1)$ &\cite{{cw2}, {melfi1},{melfi2}, {royer}}   \\
\hline 
8 & $(1,1), (1,2),(2,1), (1,4), (4,1)$ & \cite{{ramanujan}, {lahiri1}, {lahiri2},{cw2}}\\
\hline 
10 & $(1,1), (1,2), (2,1), (1,4),(4,1)$ & \cite{{ramanujan}, {lahiri1}, {lahiri2},{cw2}}\\
\hline 
12 & $(1,1)$ & \cite{{ramanujan},{lahiri1}, {lahiri2}}\\
\hline 
\end{tabular}
\end{center}

\smallskip

To illustrate our method, we prove the identities obtained in \cite[Theorems 4.2, 5.2, 6.2, 7.2]{cw2}, when $k=4,6,8,10$ and 
$(a,b) = (1,2), (2,1), (1,4), (4,1)$ using \corref{1-k-1}. In the following we express ${\mathcal D}E_k(bz) - (k/12) aE_2(az)E_k(bz)$ as a linear combination of basis elements of $M_6(ab)$. Comparing the $n$-th Fourier coefficients, the results obtained in 
Theorems 4.2, 5.2, 6.2, 7.2 of \cite{cw2} follow directly. In some cases, the cusp forms appearing in the formulas of \cite{cw2} differ from our cusp forms (which are mentioned after these expressions). 
  
\begin{align}
{\mathcal D}E_4(2z) - \frac{1}{3} E_2(z) E_4(2z) & = - \frac{1}{63} E_6(z) - \frac{20}{63} E_6(2z),  \label{21}\\
{\mathcal D}E_4(z) - \frac{2}{3} E_2(2z) E_4(z) & = -\frac{10}{63} E_6(z) - \frac{32}{63} E_6(2z),\label{22}\\
{\mathcal D}E_4(4z) - \frac{1}{3} E_2(z) E_4(4z) & = -\frac{1}{1008} E_6(z) - \frac{5}{336} E_6(2z) - \frac{20}{63} E_6(4z) + \frac{15}{2}\Delta_{6,4}(z),\label{23}\\
{\mathcal D}E_4(z) - \frac{4}{3} E_2(4z) E_4(z) & = -\frac{5}{63} E_6(z) - \frac{5}{21} E_6(2z) - \frac{64}{63} E_6(4z) -120 \Delta_{6,4}(z),\label{24}\\
{\mathcal D}E_6(2z) - \frac{1}{2} E_2(z) E_6(2z) & = -\frac{1}{170} E_8(z) - \frac{42}{85} E_8(2z) + \frac{252}{17} \Delta_{8,2}(z),\label{25}\\
{\mathcal D}E_6(z) -  E_2(2z) E_6(z) & =  -\frac{21}{85} E_8(z) - \frac{64}{85} E_8(2z) + \frac{2017}{17} \Delta_{8,2}(z),\label{26}\\
{\mathcal D}E_6(4z) - \frac{1}{2} E_2(z) E_6(4z) & =  -\frac{1}{10880}E_8(z)-\frac{63}{10880}E_8(2z)-\frac{42}{85}E_8(4z) +\frac{819}{68} \Delta_{8,2}(z)  \notag \\
& \quad + \frac{2394}{17}\Delta_{8,2}(2z),\label{27}\\
{\mathcal D}E_6(z) - 2 E_2(4z) E_6(z) & =  -\frac{21}{170}E_8(z)-\frac{63}{170}E_8(2z)-\frac{128}{85}E_8(4z) +\frac{9576}{17} \Delta_{8,2}(z)  \notag \\
&\quad + \frac{209664}{17}\Delta_{8,2}(2z),\label{28}\\
{\mathcal D}E_8(2z) - \frac{2}{3} E_2(z) E_8(2z) & =  -\frac{2}{1023} E_{10}(z) - \frac{680}{1023} E_{10}(2z) + \frac{480}{31}\Delta_{10,2}(z),\label{29}\\
{\mathcal D}E_8(z) - \frac{4}{3} E_2(2z) E_8(z) & = -\frac{340}{1023} E_{10}(z) - \frac{1024}{1023} E_{10}(2z) - \frac{7680}{31}\Delta_{10,2}(z),\label{30}\\
{\mathcal D}E_8(4z) - \frac{2}{3} E_2(z) E_8(4z) & =  -\frac{1}{130944} E_{10}(z) -\frac{85}{43648} E_{10}(2z) - \frac{680}{1023} E_{10}(4z) + \frac{585}{62} \Delta_{10,2}(z) \notag \\
& \quad -\frac{3240}{31} \Delta_{10,2}(2z) + \frac{105}{16}\Delta_{10,4}(z),\label{31}\\
{\mathcal D}E_8(z) - \frac{8}{3} E_2(4z) E_8(z) & =  -\frac{170}{1023} E_{10}(z) -\frac{170}{341} E_{10}(2z) - \frac{2048}{1023} E_{10}(4z) + \frac{25920}{31} \Delta_{10,2}(z) \notag \\
 &\quad -\frac{2396160}{31} \Delta_{10,2}(2z) -1680 \Delta_{10,4}(z),\label{32}\\
{\mathcal D}E_{10}(2z) - \frac{5}{6} E_2(z) E_{10}(2z) & =  -\frac{1}{1638} E_{12}(z) - \frac{682}{819} E_{12}(2z) + \frac{13860}{691} \Delta(z) + \frac{297792}{691} \Delta(2z), \label{33}\\
{\mathcal D}E_{10}(z) - \frac{5}{3} E_2(2z) E_{10}(z) & =  -\frac{341}{819} E_{12}(z) - \frac{1024}{819} E_{12}(2z) + \frac{148896}{691} \Delta(z)  + \frac{28385280}{691} \Delta(2z),\label{34}\\
%
{\mathcal D}E_{10}(4z) - \frac{5}{6} E_2(z) E_{10}(4z) & = -\frac{1}{1677312}E_{12}(z)- \frac{341}{559104}E_{12}(2z)-\frac{682}{819}E_{12}(4z) + \frac{599445}{44224}\Delta(z) \notag \\
& \quad +\frac{2130975}{5528}\Delta(2z)+\frac{19999584}{691}\Delta(4z)+\frac{825}{128}\Delta_{12,4}(z),\label{35}\\
{\mathcal D}E_{10}(z) - \frac{10}{3} E_2(4z) E_{10}(z) & = -\frac{341}{1638}E_{12}(z)- \frac{341}{546}E_{12}(2z)-\frac{2048}{819}E_{12}(4z) + \frac{4999896}{691}\Delta(z)  \notag \\
& \quad +\frac{272764800}{691}\Delta(2z)+\frac{39285227520}{691}\Delta(4z)-6600 \Delta_{12,4}(z).\label{36}
\end{align}

\smallskip

We make some remarks and the theorem numbers mentioned in these remarks are from the work of Cheng and Williams \cite{cw2}. 

\begin{rmk}
Our identities for $W_{1,2}^{1,3}(n)$, $W_{2,1}^{1,3}(n)$, $W_{1,4}^{1,3}(n)$, $W_{4,1}^{1,3}(n)$ (resulting from \eqref{21}-\eqref{24}) match with the formulas obtained in Theorem 4.2. Note that the Fourier coefficient $a(n)$ appearing in 
Theorem 4.2 is nothing but $\tau_{6,4}(n)$. Similarly, the formulas for $W_{1,2}^{1,5}(n)$, $W_{2,1}^{1,5}(n)$, $W_{1,4}^{1,5}(n)$, $W_{4,1}^{1,5}(n)$ (following from \eqref{25}-\eqref{28}) are identical with the formulas in Theorem 5.2 (i), (iii), (iv), (vi). In this case the Fourier coefficient $b(n)$ in their formulas is the $n$-th Fourier coefficient of the newform $\Delta_{8,2}(z)$. 
Formulas for $W_{1,2}^{1,9}(n)$, $W_{2,1}^{1,9}(n)$ obtained from \eqref{33}, \eqref{34} are the same formulas given as in 
Theorem 7.2 (i), (v). 
\end{rmk}

\begin{rmk}
From the identity \eqref{29} we derive the following. For a natural number $n$,   
\begin{equation}
W_{1,2}^{1,7}(n) = \frac{1}{14880} \sigma_9(n) + \frac{17}{744} \sigma_9(n/2)  + \frac{(2-3n)}{48} \sigma_7(n/2) - \frac{1}{480} \sigma(n) + \frac{1}{496}\tau_{10,2}(n). 
\end{equation}
Comparing the above identity with \cite[Theorem 6.2 (i)]{cw2}, we get 
\begin{equation*}
 c(n) + 32 d(n) = \tau_{10,2}(n), 
\end{equation*}
the $n$-th Fourier coefficient of the normalized newform $\Delta_{10,2}(z)$, where $c(n)$ and $d(n)$ are Fourier 
coefficients of certain eta-products \cite[Eq.(6.1), Eq.(6.2)]{cw2}. Therefore, we have 
\begin{equation}
\Delta_{10,2}(z) = \eta^{16}(z) \eta^4(2z) + 32 \eta^8(z) \eta^4(2z) \eta^8(4z),
\end{equation}
where $\eta(z) = q^{1/24}\prod_{n\ge 1}(1-q^n)$, is the Dedekind eta-function, $q= e^{2\pi iz}$. 
We obtain a similar conclusion from our convolution sum $W_{2,1}^{1,7}(n)$ obtained from \eqref{30}, which corresponds to 
Theorem 6.2 (ii). However, the formulas arising from \eqref{31} and  \eqref{32} for the sums $W_{1,4}^{1,7}(n)$ and 
$W_{4,1}^{1,7}(n)$ do not match with the formulas in Theorem 6.2 (v) and (viii). The reason is that in their work, Cheng and Williams use different cusp forms $C(z)$, $D(z)$ and $E(z)$ (refer \cite[Eqs. (6.1), (6.2), (6.3)]{cw2}. Comparing our formulas and their formulas, we obtain the following relations for all natural numbers $n$:
\begin{align*}
\frac{529}{253952} c(n) + \frac{625}{15872} d(n) + \frac{109}{992} e(n) & = 
\frac{39}{31744} \tau_{10,2}(n) - \frac{27}{1984} \tau_{10,2}(n/2) + \frac{7}{8192} \tau_{10,4}(n),\\
-\frac{109}{3968} c(n) - \frac{625}{248} d(n) - \frac{1058}{31} e(n) & = 
\frac{27}{992} \tau_{10,2}(n) - \frac{78}{31} \tau_{10,2}(n/2) - \frac{7}{128} \tau_{10,4}(n),
\end{align*}
where $c(n), d(n), e(n)$ are the $n$-th Fourier coefficients of $C(z), D(z)$ and $E(z)$ respectively. 
\end{rmk}

\smallskip

\begin{rmk}
We now consider the sums $W_{1,4}^{1,9}(n)$, $W_{4,1}^{1,9}(n)$. In the notation of \cite{cw2}, these are denoted respectively by $U_{9,1}(n)$ and $U_{1,9}(n)$. Using the identities \eqref{35} and \eqref{36} and comparing the $n$-th Fourier coefficients, we obtain for all integers $n\ge 1$, 

\begin{align*}
W_{1,4}^{1,9}(n)  &= \frac{1}{93401088} \sigma_{11}(n) + \frac{31}{2830336} \sigma_{11}(n/2) + \frac{31}{2073} \sigma_{11}(n/4)+ \frac{(5-6n)}{120} \sigma_{9}(n/4) + \frac{1}{264} \sigma(n)\\
&\qquad - \frac{3633}{1415168}\tau(n) - \frac{12915}{176896}\tau(n/2) - \frac{18939}{3455}\tau(n/4) - \frac{5}{4096} \tau_{12,4}(n),\\
W_{4,1}^{1,9}(n) &=  \frac{31}{33168} \sigma_{11}(n) + \frac{31}{11056} \sigma_{11}(n/2) + \frac{256}{22803} \sigma_{11}(n/4)+ \frac{(10-3n)}{240} \sigma_{9}(n)+ \frac{1}{264} \sigma(n/4)  \\
& \qquad - \frac{18939}{55280}\tau(n) - \frac{12915}{691}\tau(n/2) - \frac{1860096}{691}\tau(n/4)+ \frac{5}{16} \tau_{12,4}(n).\\
\end{align*}
In \cite{cw2}, Cheng and Williams used the Fourier coefficient $f(n)$ of the eta quotient $F(z)=\frac{\eta^{32}(4z)}{\eta^8(z)}$ 
(apart from the Eisenstein series $E_{12}(z)$ and $\Delta(z)$)  to express the convolution sums $U_{9,1}(n)$ and 
$U_{1,9}(n)$. (It is to be noted that the modular form $F(z)$ is  not a cusp form.)  However, we use the Eisenstein series $E_{12}(z)$ and the two cusp forms  $\Delta(z)$ and $\Delta_{12, 4}(z)$ to express our formulas. Therefore, our formulas presented above do not match with the formulas obtained in \cite{cw2}. Noting that 
\begin{align*}
F(z) &=\frac{1}{4293918720}E_{12}(z) -\frac{1} {4293918720}E_{12}(2z)-\frac{347}{67928064}\Delta(z) \\
& \qquad -\frac{475}{2830336}\Delta(2z)-\frac{5}{48}\Delta(4z)+\frac{1}{196608}\Delta_{12, 4}(z),
\end{align*}
and comparing the $n$-th coefficients one can get an expression for $f(n)$ in terms of $\sigma_{11}(n)$, $\tau(n)$ and $\tau_{12,4}(n)$. Substituting this expression (for $f(n)$) in 
Theorem 7.2, (vi) and (x),  we obtain $U_{9,1}(n)= W_{1,4}^{1,9}(n) $ and $U_{1,9}(n) = W_{4,1}^{1,9}(n) $.
\end{rmk}

\smallskip

In the following, we give some new formulas for $W_{a,b}^{1,k-1}(n)$. Specifically, we give below the formulas corresponding to the cases $(a,b) = (1,3), (3,1), (1,6), (6,1), (2,3), (3,2)$,  when $k= 4,6,8,10$. To get these formulas we make use of the basis for $M_{k+2}(6)$ given in Table 1. Using our method, the following formulas hold for all $n\in {\mathbb N}$.   


\begin{eqnarray*}
 W_{1,3}^{1,3}(n) &=& \frac{1}{1040} \sigma_5(n) + \frac{9}{104} \sigma_5(n/3)  + \frac{(1-3n)}{24}\sigma_{3}(n/3) - \frac{1}{240}\sigma(n) + \frac{1}{312} \tau_{6,3}(n)\label{W1313},\nonumber \\
W_{3,1}^{1,3}(n) &=& \frac{1}{104} \sigma_5(n) + \frac{81}{1040} \sigma_5(n/3)  + \frac{(1-n)}{24} \sigma_{3}(n)  - \frac{1}{240}\sigma(n/3) - \frac{1}{104} \tau_{6,3}(n),\nonumber \\
W_{1,6}^{1,3}(n) &=& \frac{1}{21840} \sigma_5(n) + \frac{1}{1092} \sigma_5(n/2) + \frac{3}{728} \sigma_5(n/3) + \frac{15}{182} \sigma_5(n/6) + \frac{(1-3n)}{24} \sigma_{3}(n/6)  \nonumber \\
&&- \frac{1}{240}\sigma(n) + \frac{ 1}{468} \tau_{6,3}(n) + \frac{7}{468} \tau_{6,3}(n/2) + \frac{1}{504} \tau_{6,6}(n) ,\nonumber 
\end{eqnarray*}

\begin{eqnarray*}
W_{6,1}^{1,3}(n) &=&\frac{5}{2184} \sigma_5(n) + \frac{2}{273} \sigma_5(n/2) + \frac{27}{1456} \sigma_5(n/3) + \frac{27}{455} \sigma_5(n/6) + \frac{(2-n)}{48} \sigma_{3}(n) \nonumber \\
&&- \frac{1}{240}\sigma(n/6)  - \frac{7}{624} \tau_{6,3}(n) - \frac{4}{39} \tau_{6,3}(n/2) - \frac{1}{84} \tau_{6,6}(n)  ,\nonumber \\
W_{2,3}^{1,3}(n) &=&\frac{1}{4368} \sigma_5(n) + \frac{1}{1365} \sigma_5(n/2) + \frac{15}{728} \sigma_5(n/3) + \frac{6}{91} \sigma_5(n/6) + \frac{(2-3n)}{48} \sigma_{3}(n/3) \nonumber \\
&&- \frac{1}{240}\sigma(n/2) + \frac{ 7}{1872} \tau_{6,3}(n) +  \frac{4}{117} \tau_{6,3}(n/2) - \frac{1}{252} \tau_{6,6}(n),\nonumber \\
W_{3,2}^{1,3}(n) &=&\frac{1}{2184} \sigma_5(n) + \frac{5}{546} \sigma_5(n/2) + \frac{27}{7280} \sigma_5(n/3) + \frac{27}{364} \sigma_5(n/6) + \frac{(1-n)}{24} \sigma_{3}(n/2) \nonumber \\
&&- \frac{1}{240}\sigma(n/3) - \frac{ 1}{156} \tau_{6,3}(n) -  \frac{7}{156} \tau_{6,3}(n/2) + \frac{1}{168} \tau_{6,6}(n),
\nonumber\\
\!\!\!\! W_{1,3}^{1,5}(n) \!\!\!\!&=&\!\!\!\! -\frac{1}{19680} \sigma_7(n)  - \frac{273}{6560} \sigma_7(n/3) + \frac{(1-2n)}{24}\sigma_5(n/3) + \frac{1}{504}\sigma(n)  + \frac{7}{3280} \tau_{8,3}(n),\\
\!\!\!\! W_{3,1}^{1,5}(n) \!\!\!\!&=&\!\!\!\!-\frac{91}{19680} \sigma_7(n)  - \frac{243}{6560} \sigma_7(n/3)  + \frac{(3-2n)}{72} \sigma_5(n) + \frac{1}{504} \sigma(n/3) + \frac{63}{3280} \tau_{8,3}(n),\\
W_{1,6}^{1,5}(n) \!\!\!\!&=&\!\!\!\! \frac{1}{1756440}\sigma_7(n)+\frac{1}{20910}\sigma_7(n/2)+\frac{13}{27880}\sigma_7(n/3)+\frac{273}{6970}\sigma_7(n/6) + \frac{(1-2n)}{24}\sigma_5(n/6)\\
&&+ \frac{1}{504}\sigma(n)-\frac{1}{1632}\tau_{8,2}(n)-\frac{1}{544}\tau_{8,2}(n/3)-\frac{1}{1476}\tau_{8,3}(n)-\frac{1}{738}\tau_{8,3}(n/2)-\frac{1}{1440}\tau_{8,6}(n),\\
W_{6,1}^{1,5}(n) \!\!\!\!&=&\!\!\!\! \frac{91}{83640}\sigma_7(n)+\frac{104}{31365}\sigma_7(n/2)+\frac{243}{27880}\sigma_7(n/3)+\frac{648}{24395}\sigma_7(n/6) + \frac{(3-n)}{72}\sigma_5(n) \\
&&+ \frac{1}{504}\sigma(n/6)-\frac{1}{1224}\tau_{8,2}(n)-\frac{243}{136}\tau_{8,2}(n/3)-\frac{1}{328}\tau_{8,3}(n)-\frac{16}{41}\tau_{8,3}(n/2)-\frac{1}{40}\tau_{8,6}(n),\\
W_{2,3}^{1,5}(n) \!\!\!\!&=&\!\!\!\! \frac{1}{83640}\sigma_7(n)+\frac{8}{219555}\sigma_7(n/2)+\frac{273}{27880}\sigma_7(n/3)+\frac{104}{3485}\sigma_7(n/6) + \frac{(1-n)}{24}\sigma_5(n/3)\\
&&+ \frac{1}{504}\sigma(n/2)-\frac{1}{408}\tau_{8,2}(n)-\frac{1}{136}\tau_{8,2}(n/3)-\frac{1}{2952}\tau_{8,3}(n)-\frac{16}{369}\tau_{8,3}(n/2)+\frac{1}{360}\tau_{8,6}(n),\\
W_{3,2}^{1,5}(n) \!\!\!\!&=&\!\!\!\! \frac{13}{250920}\sigma_7(n)+\frac{91}{20910}\sigma_7(n/2)+\frac{81}{195160}\sigma_7(n/3)+\frac{243}{6970}\sigma_7(n/6) + \frac{(3-2n)}{72}\sigma_5(n/2) \\
&&+ \frac{1}{504}\sigma(n/3)-\frac{1}{4896}\tau_{8,2}(n)-\frac{243}{544}\tau_{8,2}(n/3)-\frac{1}{164}\tau_{8,3}(n)-\frac{1}{82}\tau_{8,3}(n/2)+\frac{1}{160}\tau_{8,6}(n),\\
W_{1,3}^{1,7}(n)  &=& \frac{1}{322080} \sigma_9(n) \!+\! \frac{123}{5368} \sigma_9(n/3)  + \frac{(2-3n)}{48} \sigma_7(n/3) - \frac{1}{480}\sigma(n) + \frac{1}{1647}\tau_{10,3;1}(n) \\
&& + \frac{7}{4752}\tau_{10,3;2}(n),\\
W_{3,1}^{1,7}(n) &=& \frac{41}{16104} \sigma_9(n) + \frac{2187}{107360} \sigma_9(n/3)  + \frac{(2-n)}{48} \sigma_7(n) - \frac{1}{480}\sigma(n/3) + \frac{1}{61}\tau_{10,3;1}(n)\\
&& - \frac{7}{176}\tau_{10,3;2}(n) ,\\
\end{eqnarray*}
\begin{eqnarray*}
W_{1,6}^{1,7}(n) &=& \frac{1}{109829280} \sigma_9(n) + \frac{17}{5491464} \sigma_9(n/2) + \frac{123}{1830488} \sigma_9(n/3) - \frac{10455}{457622} \sigma_9(n/6) \\
&&+ \frac{(2-3n)}{48} \sigma_7(n/6) - \frac{1}{480}\sigma(n) + \frac{1}{1488} \tau_{10,2}(n) + \frac{21}{248} \tau_{10,2}(n/3) + \frac{7}{13176} \tau_{10,3;1}(n)  \\
&&+ \frac{23}{1647} \tau_{10,3;1}(n/2) + \frac{7}{13068} \tau_{10,3;2}(n) - \frac{91}{26136} \tau_{10,3;2}(n/2) + \frac{1}{2904} \tau_{10,6}(n) ,\\
W_{6,1}^{1,7}(n) &=& \frac{3485}{5491464} \sigma_9(n) + \frac{1312}{686433}\sigma_9(n/2) + \frac{37179}{7321952} \sigma_9(n/3) + \frac{17496}{1144055} \sigma_9(n/6) \\
&&+ \frac{(4-n)}{96} \sigma_7(n) - \frac{1}{480}\sigma(n/6) -\frac{7}{93} \tau_{10,2}(n) - \frac{2187}{62} \tau_{10,2}(n/3) + \frac{23}{244} \tau_{10,3;1}(n) \\
&& + \frac{224}{61} \tau_{10,3;1}(n/2) +\frac{91}{3872} \tau_{10,3;2}(n) - \frac{448}{121} \tau_{10,3;2}(n/2) - \frac{9}{121} \tau_{10,6}(n),\\
W_{2,3}^{1,7}(n) &=& \frac{17}{21965856} \sigma_9(n) + \frac{8}{3432165} \sigma_9(n/2) + \frac{10455}{1830488} \sigma_9(n/3) + \frac{3936}{228811} \sigma_9(n/6)\\
&&+ \frac{(4-3n)}{96} \sigma_7(n/3) - \frac{1}{480}\sigma(n/2) -\frac{1}{186} \tau_{10,2}(n) - \frac{21}{31} \tau_{10,2}(n/3) + \frac{23}{6588} \tau_{10,3;1}(n)  \\
&&+ \frac{224}{1647} \tau_{10,3;1}(n/2) - \frac{91}{104544} \tau_{10,3;2}(n) + \frac{448}{3267} \tau_{10,3;2}(n/2) + \frac{1}{363} \tau_{10,6}(n),\\
W_{3,2}^{1,7}(n) &=& \frac{41}{5491464} \sigma_9(n) + \frac{3485}{1372866} \sigma_9(n/2) + \frac{2187}{36609760} \sigma_9(n/3) + \frac{37179}{1830488} \sigma_9(n/6)\\
&&+ \frac{(2-n)}{48} \sigma_7(n/2) - \frac{1}{480}\sigma(n/3) +\frac{7}{744} \tau_{10,2}(n) + \frac{2187}{496} \tau_{10,2}(n/3) + \frac{7}{488} \tau_{10,3;1}(n) \\
&& + \frac{23}{61} \tau_{10,3;1}(n/2) - \frac{7}{484} \tau_{10,3;2}(n) + \frac{91}{968} \tau_{10,3;2}(n/2) - \frac{9}{968} \tau_{10,6}(n),\\
W_{1,3}^{1,9}(n)  &=&\!\!\! \frac{1}{4438984} \sigma_{11}(n) + \frac{6039}{403544} \sigma_{11}(n/3) +\frac{(5-6n)}{120} \sigma_{9}(n/3) + \frac{1}{264} \sigma(n) - \frac{98}{35241}\tau(n) \\
&&+ \frac{29034}{58735}\tau(n/3) - \frac{5}{4964} \tau_{12,3}(n),\\
W_{3,1}^{1,9}(n) &=&\!\!\! \frac{671}{403544} \sigma_{11}(n) + \frac{59049}{4438984} \sigma_{11}(n/3) + \frac{(5-2n)}{120} \sigma_{9}(n) + \frac{1}{264} \sigma(n/3) + \frac{3226}{58735}\tau(n)  \\
&&- \frac{1928934}{11747}\tau(n/3) - \frac{405}{4964} \tau_{12,3}(n),\\
W_{1,6}^{1,9}(n) &=&\!\!\! \frac{1}{6059213160} \sigma_{11}(n) + \frac{31}{137709390} \sigma_{11}(n/2) + \frac{2013}{183612520} \sigma_{11}(n/3) + \frac{686433}{45903130} \sigma_{11}(n/6) \\
&& + \frac{(5-6n)}{120} \sigma_{9}(n/6) + \frac{1}{264} \sigma(n)- \frac{1715}{1691568}\tau(n) - \frac{2303}{105723}\tau(n/2) + \frac{11291}{187952}\tau(n/3)\\
&& + \frac{227433}{58735}\tau(n/6)-\frac{395}{431868} \tau_{12,3}(n) + \frac{9835}{215934} \tau_{12,3}(n/2) - \frac{7}{11745} \tau_{12,6;1}(n) \\
&& - \frac{641}{1235520} \tau_{12,6;2}(n) - \frac{1}{1344} \tau_{12,6;3}(n),\\
\end{eqnarray*}
\begin{eqnarray*}
W_{6,1}^{1,9}(n) &=&\!\!\! \frac{228811}{550837560} \sigma_{11}(n) + \frac{85888}{68854695} \sigma_{11}(n/2) + \frac{610173}{183612520} \sigma_{11}(n/3) + \frac{2519424}{252467215} \sigma_{11}(n/6)\\ 
&&+ \frac{(5-n)}{120} \sigma_{9}(n) + \frac{1}{264} \sigma(n/6) +\frac{75811}{704820}\tau(n) + \frac{722624}{35241}\tau(n/2) - \frac{15109983}{46988}\tau(n/3)\\
&& - \frac{720135360}{11747}\tau(n/6) +\frac{265545}{287912} \tau_{12,3}(n) - \frac{2730240}{35989} \tau_{12,3}(n/2) - \frac{112}{145} \tau_{12,6;1}(n)  \\
&&+ \frac{35}{52} \tau_{12,6;2}(n) - \frac{27}{28} \tau_{12,6;3}(n) ,\\
W_{2,3}^{1,9}(n) &=&\!\!\! \frac{31}{550837560} \sigma_{11}(n)+ \frac{128}{757401645} \sigma_{11}(n/2) + \frac{686433}{183612520} \sigma_{11}(n/3) + \frac{257664}{22951565} \sigma_{11}(n/6)\\ 
&&+ \frac{(5-3n)}{120} \sigma_{9}(n/3) + \frac{1}{264} \sigma(n/2) -\frac{2303}{422892}\tau(n) - \frac{109760}{105723}\tau(n/2) + \frac{227433}{234940}\tau(n/3)\\
&& + \frac{2167872}{11747}\tau(n/6) +\frac{9835}{863736} \tau_{12,3}(n) - \frac{101120}{107967} \tau_{12,3}(n/2) - \frac{112}{11745} \tau_{12,6;1}(n) \\
&& - \frac{35}{4212} \tau_{12,6;2}(n) +\frac{1}{84} \tau_{12,6;3}(n) ,\\
W_{3,2}^{1,9}(n) &=&\!\!\! \frac{671}{550837560} \sigma_{11}(n) + \frac{228811}{137709390} \sigma_{11}(n/2) + \frac{19683}{2019737720} \sigma_{11}(n/3) + \frac{610173}{45903130} \sigma_{11}(n/6)\\
&&+ \frac{(5-2n)}{120} \sigma_{9}(n/2) + \frac{1}{264} \sigma(n/3) +\frac{11291}{563856}\tau(n) + \frac{75811}{176205}\tau(n/2) - \frac{11252115}{187952}\tau(n/3)\\
&& - \frac{15109983}{11747}\tau(n/6) -\frac{10665}{143956} \tau_{12,3}(n) + \frac{265545}{71978} \tau_{12,3}(n/2) - \frac{7}{145} \tau_{12,6;1}(n) \\
&&+ \frac{35}{832} \tau_{12,6;2}(n) +\frac{27}{448} \tau_{12,6;3}(n).
\end{eqnarray*}


\begin{rmk}
Formulas for $W_{1,3}^{1,3}(n)$ and $W_{3,1}^{1,3}(n)$ were evaluated by Yao and Xia \cite{yao} using $(p,k)$ parametrization method. Our formulas for these sums presented above are the same as obtained by Yao and Xia. 
Using the same method ($(p,k)$ parametrization), in \cite{kokluce}, K\"okl\"uce derived the convolution sums $W_{a,b}^{1,3}(n)$, where $(a,b) = (1,6), (6,1), (2,3), (3,2)$. Our formulas presented above match with the formulas obtained by K\"okl\"uce modulo the cusp form part. In his formulas, K\"okl\"uce used the following cusp forms given by a linear combination of eta-products: 
\begin{align*}
\sum_{n\ge 1} u_1(n) q^n & = -4 \eta(z)^5 \eta(2z)^5 \eta(3z) \eta(6z) + 105 \eta(z)^6 \eta(3z)^6 + 972 \eta(z) \eta(2z) 
\eta(3z)^5 \eta(6z)^5, \\
\sum_{n\ge 1} u_2(n) q^n & = -72 \eta(z)^5 \eta(2z)^5 \eta(3z) \eta(6z) + 70 \eta(z)^6 \eta(3z)^6 + 24 \eta(z) \eta(2z) 
\eta(3z)^5 \eta(6z)^5, \\
\sum_{n\ge 1} u_3(n) q^n & = 4 \eta(z)^5 \eta(2z)^5 \eta(3z) \eta(6z) + 14 \eta(z)^6 \eta(3z)^6 + 120 \eta(z) \eta(2z) 
\eta(3z)^5 \eta(6z)^5, \\
\sum_{n\ge 1} u_4(n) q^n & = -101 \eta(z)^5 \eta(2z)^5 \eta(3z) \eta(6z) + 105 \eta(z)^6 \eta(3z)^6 - 27 \eta(z) \eta(2z) 
\eta(3z)^5 \eta(6z)^5. 
\end{align*}
For these cases, our formulas involve the cusp forms $\Delta_{6,3}(z), \Delta_{6,3}(2z), \Delta_{6,6}(z)$. 
Comparing our formulas with the corresponding formulas obtained in \cite{kokluce}, we get the following expressions for $u_j(n)$, $1\le j\le 4$ in terms of $\tau_{6,3}(n)$ and $\tau_{6,6}(n)$:
\begin{align*}
u_1(n) & = 49 \tau_{6,3}(n) + 448 \tau_{6,3}(n/2) + 52 \tau_{6,6}(n),\\
3 u_2(n) & = 98 \tau_{6,3}(n) + 896\tau_{6,3}(n/2) - 104 \tau_{6,6}(n),\\
3 u_3(n) & = 28 \tau_{6,3}(n) + 112 \tau_{6,3}(n/2) + 26 \tau_{6,6}(n),\\
u_4(n) & = 56 \tau_{6,3}(n) + 392 \tau_{6,3}(n/2) - 52 \tau_{6,6}(n).
\end{align*}
From these relations, we obtain the newform $\Delta_{6,6}(z)$ in terms of eta-products, which we give below. 
\begin{equation}
\Delta_{6,6}(z) = \eta(z)^5 \eta(2z)^5 \eta(3z) \eta(6z) + 9 \eta(z) \eta(2z) \eta(3z)^5 \eta(6z)^5.
\end{equation}
Note that our  expression for this newform (given in the Appendix) is different from the above expression.
\end{rmk}


\subsection{Examples for the convolution sums $W_{a,b}(n)$} 
We begin by giving the references to the earlier works in evaluating this type of convolution sums. The list may be incomplete.

\begin{center}
{\bf Table 3}. Known convolution sums $W_{a,b}(n)$. 
\smallskip
\begin{tabular}{|c|c|c|l|}
\hline
\hline
Type of $(a,b)$ & $(a,b)$ & Level $N = ab$ & References\\
\hline
$ (1, 1)$& $(1, 1)$&$1$&\cite{{besge},{glaisher},{ramanujan}}\\
\hline
$ (1, p)$&$(1, 2), (1, 3), (1, 5), (1, 7),(1,11)$&$2,3,5,7,11,$&\cite{{aygin}, {chan-cooper},{Cho}, {cooper-toh},{huard},{lemire-williams}, {sarajevo}}\\
&$(1, 13), (1, 17), (1, 23), (1,29), (1,41)$&$13,17,23,29,41$&\\
& $(1,47), (59), (1,71)$ & $47, 59, 71$  &\\
\hline
$ (p_1, p_2)$&$(2, 3), (2, 5), (2,7), (3, 5), (3,7)$&$6,10, 14,15,21$&\cite{{aw3}, {aygin}, {cooper-ye},{rs0}, {sarajevo}}\\
\hline
$ (1, p_1 p_2)$&$(1, 6), (1, 10), (1, 14), (1, 15),(1,21)$&$6,10,14,15,21$&\cite{{aw3}, {aygin},{rs0}, {sarajevo},{royer}}\\
\hline
$ (1, p^i)$&$ (1, 4), (1, 8), (1, 9), (1, 16),$&$4,8,9,16,$&\cite{{aw5},{alaca-kesicioglu},{huard},{w1},{w2}, {tian-yao-xia}}\\
&$ (1, 25), (1, 27), (1, 32)$& $25,27,32$&\\
\hline
$ (1, p_1^i p_2^j )$&$ (1, 12), (1, 18), (1, 20),$&$12,18,20,$&\cite{{aw1},{aw4},{cooper-ye}, {okayama}, {sarajevo}}\\
&$(1, 24), (1,28), (1, 36)$& $24, 28, 36$&\\
\hline
$ (p_1^i, p_2^j )$&$ (3, 4), (2, 9), (4, 5),$&$12,18,20,$&\cite{{aw2},{aw4},{cooper-ye}, {okayama}, {sarajevo}}\\
&$ (3, 8), (4,7), (4, 9)$& $24, 28, 36$&\\
\hline
\hline 
\end{tabular}
\end{center}

The following examples are for the convolution sums obtained in \corref{wab}.  Here we mention one set of examples 
when the level of the modular forms space is 30. 
Let $n\in {\mathbb N}$, then we have the following formulas for the convolution sums $W_{a,b}(n)$ with $ab\vert 30$, $\gcd(a,b)=1$. To get these examples, we make use of the basis of $M_4(30)$ presented in Table 1, which  consists of 22 elements (for the convolution sums when $ab$ is a proper divisor of 30, we use the basis for the level $ab$, which is a subset of these basis elements). Using \corref{wab}, we get the following formulas for all natural numbers $n$:
\begin{align}
W_{15}(n) & = \frac{1}{624}\sigma_3(n)  + \frac{3}{208} \sigma_3(n/3)  + \frac{25}{624} \sigma_3(n/5) + \frac{75}{208} \sigma_3(n/15) +\frac{(5-2n)}{120} \sigma(n) \notag \\
& \quad +\frac{(1-6n)}{24}  \sigma(n/15) - \frac{1}{455} \tau_{4,5}(n)  - \frac{9}{455}\tau_{4,5}(n/3)  - \frac{1}{80} \tau_{4,15;1}(n) - \frac{1}{84}\tau_{4,15;2}(n),\label{w15}
\end{align}
\begin{align}
W_{3,5}(n) &= \frac{1}{624}\sigma_3(n)  + \frac{3}{208} \sigma_3(n/3)  + \frac{25}{624} \sigma_3(n/5) + \frac{75}{208} \sigma_3(n/15) +\frac{(5-6n)}{120} \sigma(n/3) \notag \\
&\quad  +\frac{(1-2n)}{24}  \sigma(n/5)  - \frac{1}{455} \tau_{4,5}(n) - \frac{9}{455}\tau_{4,5}(n/3)  + \frac{1}{80} \tau_{4,15;1}(n) - \frac{1}{84}\tau_{4,15;2}(n),\label{w35}\\
W_{2,15}(n)  &= \frac{1}{3120}\sigma_3(n) +  \frac{1}{780}\sigma_3(n/2) +  \frac{3}{1040}\sigma_3(n/3)  +  \frac{5}{624}\sigma_3(n/5)+ \frac{3}{260}\sigma_3(n/6) \notag \\
& + \frac{5}{156}\sigma_3(n/10) + \frac{15}{208}\sigma_3(n/15)+  \frac{15}{52}\sigma_3(n/30) +\frac{(5-2n)}{120} \sigma(n/2) +\frac{(1-3n)}{24} \sigma(n/15)\notag \\
& - \frac{3}{910}\tau_{4,5}(n)  - \frac{6}{455} \tau_{4,5}(n/2)  - \frac{27}{910} \tau_{4,5}(n/3) - \frac{54}{455} \tau_{4,5}(n/6) - \frac{1}{720} \tau_{4,6}(n)  - \frac{5}{144} \tau_{4,6}(n/5) \notag \\
& + \frac{1}{120} \tau_{4,10}(n) + \frac{3}{40} \tau_{4,10}(n/3) - \frac{3}{560} \tau_{4,15;1}(n) + \frac{3}{140} \tau_{4,15;1}(n/2) - \frac{1}{252} \tau_{4,15;2}(n) \notag \\
&  - \frac{1}{63} \tau_{4,15;2}(n/2) - \frac{1}{336} \tau_{4,30;1}(n) - \frac{1}{120} \tau_{4,30;2}(n), \\
W_{3,10}(n)  &= \frac{1}{3120}\sigma_3(n) +  \frac{1}{780}\sigma_3(n/2) +  \frac{3}{1040}\sigma_3(n/3)  +  \frac{5}{624}\sigma_3(n/5)+ \frac{3}{260}\sigma_3(n/6) \notag \\
& + \frac{5}{156}\sigma_3(n/10) + \frac{15}{208}\sigma_3(n/15)+  \frac{15}{52}\sigma_3(n/30) + \frac{(5-3n)}{120} \sigma(n/3) +\frac{(1-2n)}{24} \sigma(n/10)\notag \\
& - \frac{3}{910}\tau_{4,5}(n)  - \frac{6}{455} \tau_{4,5}(n/2) - \frac{27}{910} \tau_{4,5}(n/3) - \frac{54}{455} \tau_{4,5}(n/6) - \frac{1}{720} \tau_{4,6}(n)\notag \\
&  - \frac{5}{144} \tau_{4,6}(n/5)  - \frac{1}{120} \tau_{4,10}(n) - \frac{3}{40} \tau_{4,10}(n/3) + \frac{3}{560} \tau_{4,15;1}(n) + \frac{3}{140} \tau_{4,15;1}(n/2)
\notag \\
& - \frac{1}{252} \tau_{4,15;2}(n) - \frac{1}{63} \tau_{4,15;2}(n/2) + \frac{1}{336} \tau_{4,30;1}(n) + \frac{1}{120} \tau_{4,30;2}(n),
\\
W_{5,6}(n)  &= \frac{1}{3120}\sigma_3(n) +  \frac{1}{780}\sigma_3(n/2) +  \frac{3}{1040}\sigma_3(n/3)  +  \frac{5}{624}\sigma_3(n/5)+ \frac{3}{260}\sigma_3(n/6) \notag \\
& + \frac{5}{156}\sigma_3(n/10) + \frac{15}{208}\sigma_3(n/15)+  \frac{15}{52}\sigma_3(n/30) + \frac{(1-n)}{24} \sigma(n/5) + \frac{(5-6n)}{120} \sigma(n/6)\notag \\
& - \frac{3}{910}\tau_{4,5}(n)  - \frac{6}{455} \tau_{4,5}(n/2) - \frac{27}{910} \tau_{4,5}(n/3) - \frac{54}{455} \tau_{4,5}(n/6) - \frac{1}{720} \tau_{4,6}(n)\notag \\
&  - \frac{5}{144} \tau_{4,6}(n/5)  - \frac{1}{120} \tau_{4,10}(n) + \frac{3}{40} \tau_{4,10}(n/3) + \frac{3}{560} \tau_{4,15;1}(n) + \frac{3}{140} \tau_{4,15;1}(n/2)\notag \\
& + \frac{1}{252} \tau_{4,15;2}(n) - \frac{1}{63} \tau_{4,15;2}(n/2) + \frac{1}{336} \tau_{4,30;1}(n) - \frac{1}{120} \tau_{4,30;2}(n),
\\
W_{30}(n) &= \frac{1}{3120}\sigma_3(n) +  \frac{1}{780}\sigma_3(n/2) +  \frac{3}{1040}\sigma_3(n/3)  +  \frac{5}{624}\sigma_3(n/5)+ \frac{3}{260}\sigma_3(n/6) \notag \\
& \quad + \frac{5}{156}\sigma_3(n/10) + \frac{15}{208}\sigma_3(n/15)+  \frac{15}{52}\sigma_3(n/30) + \frac{(5-n)}{120} \sigma(n) +\frac{(1-6n)}{24} \sigma(n/30) \notag \\
& \quad 
- \frac{3}{910}\tau_{4,5}(n)  - \frac{6}{455} \tau_{4,5}(n/2) 
- \frac{27}{910} \tau_{4,5}(n/3) - \frac{54}{455} \tau_{4,5}(n/6) - \frac{1}{720} \tau_{4,6}(n) \notag \\
& \quad  - \frac{5}{144} \tau_{4,6}(n/5)  - \frac{1}{120} \tau_{4,10}(n) - \frac{3}{40} \tau_{4,10}(n/3) - \frac{3}{560} \tau_{4,15;1}(n)
 - \frac{3}{140} \tau_{4,15;1}(n/2) \notag \\
& \quad  - \frac{1}{252} \tau_{4,15;2}(n) - \frac{1}{63} \tau_{4,15;2}(n/2) 
- \frac{1}{336} \tau_{4,30;1}(n) - \frac{1}{120} \tau_{4,30;2}(n).
\end{align}

\begin{rmk}
We note that the identities for $W_{15}(n)$ and $W_{3,5}(n)$  were evaluated earlier by the first two authors in \cite[Theorem 2.2]{rs0} by a different method. However, the coefficients of the newforms $\Delta_{4, 15: 1} (z)$ and $\Delta_{4, 15: 2}(z),$ were incorrectly given in those formulas (coefficients were interchanged). The above identities \eqref{w15} and \eqref{w35} rectify this mistake. 
\end{rmk}

\begin{rmk}
In \cite{{aat}, {t1}, {t2}}, the authors consider evaluation of the sum $W_{a,b}(n)$ using the theory of modular forms for several cases of $ab$ and in \cite{t3}, the author studied the same for general $ab$. Further, in \cite{t4}, the author evaluates the convolution sums $W_{a,b}^{1,3}(n)$ for some cases of $ab$ and all the above works make use of the theory of modular forms to evaluate these sums. Our approach in this paper is different from their method. In all the works mentioned above, the authors make use of these convolution sums to find formulas for the number of representations of a natural number by certain quadratic forms in 8 or 12 variables. It is to be noted that the theory of modular forms can be used directly to find these formulas (for the representation numbers) as demonstrated in our earlier work \cite{rss}.
\end{rmk}

\subsection{Examples for the convolution sums $\sum_{al+bm=n}l\sigma(l)\sigma_{k-1}(m)$}

In this section, we give some examples for \thmref{3.4}. We take $(a,b) = (1,2), (1,3)$ when $k=4,6,8$ and for 
$k=10$, we take $(a,b) = (1,2)$. Basis for the corresponding space of modular forms are listed in Table 1. 
Our method gives the following convolution sums for all natural numbers $n$. 
\begin{align}
\sum_{l,m\in {\mathbb N}\atop{l+2m=n}} l\sigma(l) \sigma_{3}(m) & = \frac{n}{720} \sigma_5(n)  + \frac{n}{36} \sigma_5(n/2) - \frac{n^2}{40}\sigma_3(n/2) - \frac{n}{240} \sigma(n) + \frac{1}{360} \tau_{8,2}(n),\\
\sum_{l,m\in {\mathbb N}\atop{l+3m=n}} l\sigma(l) \sigma_{3}(m) &= \frac{n}{3120}\sigma_5(n) + \frac{3n}{104}\sigma_5(n/3) 
-\frac{n^2}{40}\sigma_3(n/3) -\frac{n}{240}\sigma(n) + \frac{n}{936}\tau_{6,3}(n) \notag\\
& \quad \qquad + \frac{1}{360}\tau_{8,3}(n),\\
\sum_{l,m\in {\mathbb N}\atop{l+2m=n}} l\sigma(l) \sigma_{5}(m) &= \frac{n}{8568}\sigma_7(n) + \frac{2n}{204}\sigma_7(n/2) - \frac{n^2}{84} \sigma_5(n/2) + \frac{n}{504}\sigma(n)  - \frac{n}{1632} \tau_{8,2}(n)\notag\\
&\quad  \qquad - \frac{1}{672}\tau_{10,2}(n),\\
\sum_{l,m\in {\mathbb N}\atop{l+3m=n}} l\sigma(l) \sigma_{5}(m) &= -\frac{n}{78720} \sigma_7(n) - \frac{273n}{26240}\sigma_7(n/3) -\frac{n^2}{84}\sigma_5(n/3) + \frac{n}{504}\sigma(n) + \frac{7n}{13120}\tau_{8,3}(n)\notag\\
& \quad \qquad - \frac{1}{3024}\tau_{10,3;1}(n) - \frac{1}{864}\tau_{10,3;2}(n), \\
\sum_{l,m\in {\mathbb N}\atop{l+2m=n}} l\sigma(l) \sigma_{7}(m) &= \frac{n}{74400} \sigma_9(n) + \frac{17n}{3720}\sigma_9(n/2) + \frac{n^2}{144}\sigma_7(n/2) - \frac{n}{480}\sigma(n)-\frac{n}{2480}\tau_{10,2}(n) \notag \\
&\quad \qquad +\frac{1}{600}\tau(n)+ \frac{2}{45} \tau(n/2),
\end{align}
\begin{align}
\sum_{l,m\in {\mathbb N}\atop{l+3m=n}} l\sigma(l) \sigma_{7}(m) &= \frac{n}{1610400}\sigma_9(n) + \frac{123n}{26840}\sigma_9(n/2) + \frac{n^2}{144}\sigma_7(n/2) - \frac{n}{480}\sigma(n) + \frac{n}{8235}\tau_{10,3;1}(n) \notag \\
 & \quad \qquad - \frac{7n}{23760}\tau_{10,3;2}(n) + \frac{1}{850} \tau(n) -  \frac{31}{170} \tau(n/3) + \frac{1}{2040} \tau_{12,3}(n),\\
\sum_{l,m\in {\mathbb N}\atop{l+2m=n}} l\sigma(l) \sigma_{9}(m) &=\frac{n}{547272}\sigma_{11}(n)+ \frac{31n}{12438}\sigma_{11}(n/2)- \frac{n^2}{220}\sigma_9(n/2) + \frac{n}{264}\sigma(n) - \frac{21n}{31168}\tau(n/2) \notag \\
&\quad \qquad  - \frac{1}{704} \tau_{14,2;1}(n)  -\frac{1}{576} \tau_{14,2;2}(n).
\end{align}

\smallskip

\begin{rmk}
We have used some of the formulas  for $W_{a,b}^{1,k-1}(n)$ presented in \S 3.1 to get the above expressions. 
\end{rmk}

\smallskip

\subsection{Examples for the convolution sums $\sum_{al+bm=n}l\sigma(l)\sigma(m)$} ~~In this case the corresponding  
vector space of modular forms has weight $6$. We consider the levels $2, 3, 4, 6$ and obtain the following sums. The basis elements for these levels appear in Table 1. Also, we give formulas for the cases $(a,b) = (1,2), (1,3), (1,4), (1,6), (2,3)$. For these values, the formulas for the pair $(b,a)$ can be obtained by observing the following relation valid for relatively prime integers $a$ and $b$ (for $n\in {\mathbb N}$): 

\begin{equation}\label{52}
b\sum_{l,m\in {\mathbb N}\atop{bl+am=n}} l\sigma(l) \sigma(m) = n W_{a,b}(n) - a\sum_{l,m\in {\mathbb N}\atop{al+bm =n}}
 l\sigma(l) \sigma(m).
 \end{equation}
 
 \smallskip

Let $n\in {\mathbb N}$. Then we have 

\begin{align}
\sum_{l,m\in {\mathbb N}\atop{l+2m=n}} l\sigma(l) \sigma(m) & = \frac{n}{24}\sigma_3(n) +  \frac{n}{6}\sigma_3(n/2) - \frac{(2n^2-n)}{24} \sigma(n) - \frac{n^2}{12} \sigma(n/2),\label{53}\\
\sum_{l,m\in {\mathbb N}\atop{l+3m=n}} l\sigma(l) \sigma(m) & =  \frac{n}{48}\sigma_3(n) +  \frac{3n}{16}\sigma_3(n/3) - \frac{(4n^2-3n)}{72} \sigma(n) - \frac{n^2}{12} \sigma(n/3) -\frac{1}{144} \tau_{6,3}(n),\\
\sum_{l,m\in {\mathbb N}\atop{l+4m=n}} l\sigma(l) \sigma(m) & =  \frac{n}{96}\sigma_3(n) +  \frac{n}{32}\sigma_3(n/2) +  \frac{n}{6}\sigma_3(n/4) - \frac{(n^2-n)}{24} \sigma(n) 
- \frac{n^2}{12} \sigma(n/4) -\frac{1}{96} \tau_{6,4}(n),\\
\sum_{l,m\in {\mathbb N}\atop{l+6m=n}} l\sigma(l) \sigma(m) & =  \frac{n}{240}\sigma_3(n) +  \frac{n}{60}\sigma_3(n/2) +  \frac{3n}{80}\sigma_3(n/3) +  \frac{3n}{20}\sigma_3(n/6)  - \frac{(2n^2-3n)}{72} \sigma(n) \notag \\
&\quad  - \frac{n^2}{12} \sigma(n/6) -\frac{n}{240} \tau_{4,6}(n) -\frac{1}{144} \tau_{6,3}(n)  -\frac{1}{18} \tau_{6,3}(n/2) -\frac{1}{144} \tau_{6,6}(n) ,
\end{align}
\begin{align}
\sum_{l,m\in {\mathbb N}\atop{2l+3m=n}} l\sigma(l) \sigma(m) & = \frac{n}{480}\sigma_3(n) +  \frac{n}{120}\sigma_3(n/2) +  \frac{3n}{160}\sigma_3(n/3) +  \frac{3n}{40}\sigma_3(n/6)  - \frac{(4n^2-3n)}{144} \sigma(n/2) \notag \\
&\quad  - \frac{n^2}{48} \sigma(n/3) -\frac{n}{480} \tau_{4,6}(n) -\frac{1}{288} \tau_{6,3}(n)  -\frac{1}{36} \tau_{6,3}(n/2) -\frac{1}{288} \tau_{6,6}(n).
\end{align}

\smallskip

Comparing the last two expressions, we obtain the following relation for all $n\in {\mathbb N}$.

\begin{align}
2\!\!\!\!\sum_{l,m\in {\mathbb N}\atop{2l+3m=n}} \!\!\! l\sigma(l) \sigma(m) - \!\!\!\! \sum_{l,m\in {\mathbb N}\atop{l+6m=n}} \!\!\! l\sigma(l) \sigma(m) &=  \frac{(2n^2-3n)}{72} \sigma(n)  - \frac{(4n^2-3n)}{72} \sigma(n/2)  - \frac{n^2}{24} \sigma(n/3) + \frac{n^2}{12} \sigma(n/6).
\end{align}
Note that the right-hand side of the above identity involves only the divisor function $\sigma(n)$. 

\smallskip

\begin{rmk}
The convolution sum given by \eqref{53} is the same as obtained in \cite[Theorem 4.1]{auw}. By using \eqref{52} and the sum $W_2(n)$, we can derive the sum $\displaystyle{\sum_{l,m\in {\mathbb N}\atop{2l+m=n}}} l\sigma(l) \sigma(m)$, which is exactly the second sum appearing in \cite[Theorem 4.1]{auw}.
\end{rmk}

\noindent {\bf Acknowledgements}. We have used the L-functions and Modular Forms Database (LMFDB) \cite{lmfdb}. We have also used the open-source mathematics software SAGE (www.sagemath.org) \cite{sage} to perform our calculations. 
We are grateful to the referee for pointing out an inaccuracy in our notation for the extended operator, which has been rectified. We are also thankful to his/her valuable comments and suggestions which improved the presentation. 

\noindent {\small {\bf Data Availability Statement}: Data sharing not applicable to this article as no datasets were generated or analysed during the current study.}

\section*{Appendix}

Here we give the expressions for the newforms appearing in Table 1:

\noindent 
{\bf Weight $4$:}
\begin{eqnarray*}
\Delta_{4,5}(z) &=&  \eta^4(z)\eta^4(5z) ,\\
\Delta_{4,6}(z) &=&  \eta^2(z)\eta^2(2z)\eta^2(3z)\eta^2(6z),\\
\Delta_{4,10}(z) &=& -\frac{3}{2}\Delta_{4,5}(z)  -4 \Delta_{4,5}(2z) + \frac{5}{2} \frac{\eta^2(z)\eta^8(5z)}{\eta^2(10z)},\\
\Delta_{4,15;1}(z) &=& \Delta_{4,5}(z) + 9 \Delta_{4,5}(3z) +5 \Delta^2_{2,15}(z) + 2 \frac{\eta^5(z)\eta^5(15z)}{\eta(3z)\eta(5z)},\\
\end{eqnarray*}
where $\Delta_{2,15}(z) ~=~\eta(z) \eta(3z) \eta(5z) \eta(15z)$,
\begin{eqnarray*}
\Delta_{4,15;2}(z) &=& \Delta_{4,5}(z) + 9 \Delta_{4,5}(3z) + 7 \Delta^2_{2,15}(z),\\
\Delta_{4,30;1}(z) &=& q - 2q^2 + 3q^3 + 4q^4 + 5q^5 - 6q^6 + 32 q^7 - 8q^8 + 9q^9 - 10q^{10} - 60q^{11} + 12q^{12}\\
&& - 34q^{13} - 64q^{14} + 15q^{15} + 16q^{16} + 42q^{17} - 18q^{18} - 76q^{19} + 20q^{20} + 96q^{21} \\
&& + 120q^{22} - 24q^{24} + 25q^{25} + O(q^{26}),
\end{eqnarray*}
\begin{eqnarray*}
\Delta_{4,30;2}(z) &=& q + 2q^2 + 3q^3 + 4q^4 - 5q^5 + 6q^6 - 4 q^7 + 8q^8 + 9q^9 - 10q^10 - 48q^{11} + 12q^{12}\\
&& + 2q^{13} - 8q^{14} - 15q^{15} + 16q^{16} - 114q^{17} + 18q^{18} + 140q^{19} - 20q^{20} - 12 q^{21}\\
&&  - 96q^{22} + 72q^{23} + 24q^{24} + 25q^{25} + O(q^{26}).
\end{eqnarray*}


\noindent 
{\bf Weight $6$:}
\begin{eqnarray*}
\Delta_{6, 3} (z) &=& \eta^6(z)\eta^6(3z),\\
\Delta_{6, 4} (z) &=& \eta^{12}(2z),\nonumber \\
\Delta_{6, 6} (z) &=& \frac{-7}{480} \left\{ \frac{50}{1911} E_6(z) + \frac{160}{1911} E_6(2z) + \frac{135}{637} E_6(3z) + \frac{432}{637} E_6(6z) + \frac{840}{13} \Delta_{6,3}(z)  \right. \nonumber \\
&&  \left. \hspace{5.8cm} + \frac{7680}{13} \Delta_{6,3}(2z)  + \frac{1}{2} DE_4(z) - E_2(6z)E_4(z)\right\}.  \nonumber 
\end{eqnarray*}


\noindent 
{\bf Weight $8$:}
\begin{eqnarray*}
\Delta_{8, 2}(z) \!\!\! &=&\!\!\! \eta^{8}(z)\eta^8(2z),\qquad \\
\Delta_{8, 3}(z) \!\!\! &=&\!\!\! \eta^{12}(z)\eta^4(3z)+  81\eta^6(z)\eta^4(3z)\eta^6(9z)}+18 \eta^9(z){\eta^4(3z)\eta^3(9z),\qquad \\
\Delta_{8, 6}(z) \!\!\! &=&\!\!\! \frac{1}{240}(E_4(z)E_4(6z)-E_4(2z)E_4(3z)).
\end{eqnarray*}


\noindent 
{\bf Weight $10$:}
\begin{eqnarray*}
\Delta_{10,2}(z) &=&\!\!\! \frac{31}{63}\frac{\eta^{16}(2z)}{\eta^{8}(z)}E_6(z) \!-\! \frac{4}{2079}E_{10}(z) \!+\! \frac{4}{2079} E_{10}(2z),\\ 
\Delta_{10,3;1}(z) &=& - \frac{45}{17248}E_{10}(z) + \frac{45}{17248}E_{10}(3z) + \frac{3355}{240 \times 5292} \Big( E_4(z) - E_4(3z) \Big)  E_6(z) \\
&& - \frac{61}{189}\Delta_{6,3}(z)E_4(z),\\
\Delta_{10,3;2}(z) &=& - \frac{9}{4312}E_{10}(z) + \frac{9}{4312} E_{10}(3z) + \frac{671}{240 \times 1323} \Big( E_4(z) - E_4(3z) \Big)  E_6(z) \\
&& - \frac{11}{189}\Delta_{6,3}(z)E_4(z),
\end{eqnarray*}
\begin{eqnarray*}
\Delta_{10,4}(z) &=& q + 228q^3 - 666q^5 - 6328q^7 + 32301q^9 - 30420q^{11} - 32338q^{13} - 151848q^{15}\\
&& + 590994q^{17} + 34676q^{19} - 1442784q^{21} + 1048536q^{23} - 1509569q^{25} + O(q^{26}),\\
\Delta_{10,6}(z) &=& -\frac{143}{243} \Delta_{10,2}(z) -297 \Delta_{10,2}(3z) + \frac{11}{9} \Delta_{10,3;1}(z)  +  \frac{1408}{27} \Delta_{10,3;1}(2z)+\frac{7}{18} \Delta_{10,3;2}(z) \\
&& -\frac{896}{27} \Delta_{10,3;3}(2z)  - \frac{11}{486} \Delta_{4,6}(z)E_6(z).
\end{eqnarray*}


\noindent 
{\bf Weight $12$:}
\begin{eqnarray*}
\Delta_{12,3}(z) &=& \frac{98}{81} \Delta(z) - \frac{275562}{81}\Delta(3z) -\frac{17}{81} \Delta_{6,3}(z)E_6(z), \\
\Delta_{12,4}(z) &=& q - 516q^3 - 10530q^5 + 49304q^7 + 89109q^9 - 309420q^{11} - 1723594q^{13}\\
&& + 5433480q^{15} - 2279502q^{17} + 4550444q^{19} - 25440864q^{21} - 7282872q^{23} \\
&& + 62052775q^{25} + O(q^{26}),
\end{eqnarray*}
\begin{eqnarray*}
\Delta_{12,6;1}(z) &=&\!\!\!-\frac{203}{323}\Delta(z) - \frac{928}{323}\Delta(2z) - \frac{147987}{323}\Delta(3z) - \frac{676512}{323}\Delta(6z) - \frac{2727}{2261}\Delta_{12,3}(z)\\
&& + \frac{21600}{323}\Delta_{12,3}(2z) + \frac{783}{266}\Delta_{6,3}(z)E_6(6z) - \frac{29}{266}\Delta_{6,3}(z)E_6(2z),\\
\!\!\!\!\!\!\!\Delta_{12,6;2}(z) &=&\!\!\!\!-\frac{16051}{16150}\Delta(z) \!-\! \frac{153728}{8075}\Delta(2z) \!+\! \frac{4781511}{16150}\Delta(3z) \!+\! \frac{340028928}{8075}\Delta(6z) \!-\! \frac{168021}{129200}\Delta_{12,3}(z)~~ \\
&& + \frac{27648}{323}\Delta_{12,3}(2z) + \frac{10989}{3325}\Delta_{6,3}(z)E_6(6z) - \frac{27}{3325}\Delta_{6,3}(z)E_6(2z) - \frac{1}{400}\Delta_{4,6}(z)E_8(z),\\
\Delta_{12,6;3}(z) &=&\!\!\!-\frac{679}{5814}\Delta(z) + \frac{112384}{8721}\Delta(2z) + \frac{311283}{646}\Delta(3z) - \frac{11477376}{323}\Delta(6z) + \frac{6223}{5168}\Delta_{12,3}(z)\\
&& - \frac{25600}{323}\Delta_{12,3}(2z) - \frac{27}{133}\Delta_{6,3}(z)E_6(6z) + \frac{407}{3591}\Delta_{6,3}(z)E_6(2z) + \frac{1}{432}\Delta_{4,6}(z)E_8(z).\\
\end{eqnarray*}


\begin{thebibliography}{10}
\bibitem{glaisher}
 J.W.L. Glaisher, {\em On the square of the series in which the coefficients are the sums of the divisors of the exponents}, Mess. Math. {\bf 14}, (1885) 156–163.
\end{thebibliography}
\end{document}